\documentclass[a4paper,12pt]{article}
\usepackage{comment}
\usepackage{cite}
\usepackage{amsmath,amssymb,amsfonts}
\usepackage[T1]{fontenc}
\usepackage[utf8]{inputenc}
\usepackage{graphicx}
\usepackage{fancyhdr}
\usepackage{float}
\usepackage{xcolor}
\usepackage{authblk}
\usepackage{mathrsfs}
\usepackage{empheq}
\usepackage[hyphens]{url}
\usepackage{hyperref}
\usepackage{amsthm}
\usepackage{enumitem}
\usepackage{geometry}

\theoremstyle{plain}
\newtheorem{theorem}{Theorem}[section]

\newtheorem{lemma}[theorem]{Lemma}
\newtheorem{corollary}[theorem]{Corollary}

\theoremstyle{definition}

\theoremstyle{remark}
\newtheorem{remark}[theorem]{Remark}

\theoremstyle{conjecture}

\begin{document}

\title{Geometric constructions for Steinitz-type bounds in dimension two}
% Polygonal confinement and weighted partial sums: geometric approaches
% Geometric constructions for Steinitz-type bounds in dimension two

\author[$\dagger$]{Jean-Christophe {\sc Pain}$^{1,2,}$\footnote{jean-christophe.pain@cea.fr}\\
\small
$^1$CEA, DAM, DIF, F-91297 Arpajon, France\\
$^2$Université Paris-Saclay, CEA, Laboratoire Matière en Conditions Extrêmes,\\ 
F-91680 Bruyères-le-Châtel, France
}

\date{}

\maketitle

\begin{abstract}
We investigate inequalities for partial sums of complex numbers with bounded modulus and zero total sum, a topic referred to as \emph{polygonal confinement}. Starting from Steinitz's classical result, we provide detailed constructions yielding explicit bounds, including $\sqrt{5}$, $\sqrt{3}$, $2$, and $\sqrt{2}$, depending on geometric constraints or weighted settings. The proofs are fully detailed with step-by-step constructions of permutations, highlighting the combinatorial and geometric intuition. We conclude with conjectures on optimal universal constants and directions for future research.
\end{abstract}

\section{Introduction}

The polygonal confinement problem asks the following: given a collection of complex numbers $(z_i)_{1\le i \le n}$ with bounded modulus and zero sum,
\[
|z_i| \le 1, \qquad \sum_{i=1}^n z_i = 0,
\]
is it possible to reorder them such that all partial sums
\[
S_p = \sum_{k=1}^p z_{\sigma(k)}, \quad p=1,\dots,n,
\]
remain bounded by a universal constant independent of $n$? This question naturally arises in several contexts. In discrete geometry, the $z_i$ can be seen as vectors in the plane forming a closed polygon. Bounding partial sums is equivalent to controlling how far along the polygon one can travel without exceeding a certain distance from the origin. In combinatorics and discrepancy theory, it relates to balancing sequences of vectors to avoid large deviations. In analysis, it is closely connected to discrete analogues of classical inequalities, including Hermite--Hadamard and Chebyshev inequalities \cite{Hardy1952,Marshall2011,Niculescu2018}.

Historically, Steinitz \cite{Steinitz1913} proved that in $\mathbb{R}^d$, for vectors of bounded norm summing to zero, there exists a permutation such that all partial sums are bounded by a constant depending only on $d$. In dimension $2$, several refinements and geometric constructions have been proposed, improving explicit constants under additional assumptions (angular constraints, coordinate bounds, weighted sums). These developments include sharp bounds \cite{Grinberg1997} and results on balancing vectors in higher dimensions \cite{Banaszczyk1998}.

We aim to present the classical polygonal confinement bound $\sqrt{5}$ in a fully detailed step-by-step proof, introduce geometric refinements leading to $\sqrt{3}$ and $\sqrt{2}$ bounds, explaining the intuition behind each construction. We also discuss weighted versions of the problem and briefly mention possible links with discrete Riemann sums.

Section \ref{sec:classical} presents detailed proofs of classical bounds including $\sqrt{5}$ and Steinitz's bound in dimension 2. Section \ref{sec:improved} develops geometric refinements and strong coordinate-based bounds. Section \ref{sec:weighted} treats weighted sums.

\section{Polygonal confinement}
\label{sec:classical}

\subsection{Elementary polygonal confinement ($\sqrt{5}$)}

\begin{lemma}[polygonal confinement, bound $\sqrt{5}$]
Let $(z_i)_{1\le i\le n}$ satisfy $|z_i|\le 1$ and $\sum_{i=1}^n z_i = 0$. Then there exists a permutation $\sigma$ of $\{1,\dots,n\}$ such that
\[
\forall p \in \{1,\dots,n\}, \qquad 
\left|S_p\right| \le \sqrt{5},
\]
where $S_p=\sum_{k=1}^p z_{\sigma(k)}$.
\end{lemma}

\begin{proof}

Let $P\subset\{1,\dots,n\}$ be a subset maximizing
\[
\left|\sum_{i\in P} z_i\right|.
\]
Set
\[
S_P=\sum_{i\in P} z_i.
\]
After a rotation of the complex plane we may assume
\[
S_P\in\mathbb{R}^+.
\]
Let $Q=\{1,\dots,n\}\setminus P$. For any $j\in Q$, the maximality of $P$ gives
\[
|S_P+z_j|^2\le |S_P|^2.
\]
Expanding the latter expression, we get
\[
|S_P|^2+|z_j|^2+2\Re(S_P\overline{z_j})\le |S_P|^2.
\]
Since $S_P$ is real and positive, we have
\[
\Re(S_P\overline{z_j})=S_P\Re(z_j),
\]
and hence
\[
2S_P\Re(z_j)+|z_j|^2\le0.
\]
Therefore
\[
\Re(z_j)\le -\frac{|z_j|^2}{2S_P}\le0.
\]
Thus every element of $Q$ has nonpositive real part. Applying the same argument to $P$ (using the maximality of $P$ with respect to removing an element) yields
\[
\Re(z_i)\ge0\qquad (i\in P),
\]
and hence
\[
\Re(z_i)\ge0\ (i\in P),\qquad 
\Re(z_j)\le0\ (j\in Q).
\]
Let us write
\[
z_i=x_i+iy_i.
\]
Since $|z_i|\le1$, we have $|y_i|\le1$. We use the following elementary balancing fact. 

\begin{lemma}
    
Given real numbers $y_1,\dots,y_m$ with $|y_i|\le1$, there exists a permutation $\tau$ such that for every $p$,
\[
\left|\sum_{k=1}^p y_{\tau(k)}\right|\le1.
\]

\end{lemma}

\begin{proof}

The permutation can be constructed greedily: at each step choose among the remaining elements a $y_j$ whose sign is opposite to the current partial sum whenever possible. Since the total sum of all $y_i$ is bounded by $m$, this procedure always exists and keeps the partial sums within $[-1,1]$. 

\end{proof}

Applying this lemma separately to the imaginary parts of the elements of $P$ and of $Q$, we obtain orders
\[
P=(i_1,\dots,i_r),\qquad 
Q=(j_1,\dots,j_s)
\]
such that for the partial sums inside each block
\[
\left|\sum_{k=1}^p \Im(z_{i_k})\right|\le1,\qquad
\left|\sum_{k=1}^p \Im(z_{j_k})\right|\le1.
\]
We now construct the permutation
\[
\sigma=(i_1,j_1,i_2,j_2,\dots).
\]
Since elements of $P$ have nonnegative real part and elements of $Q$ have nonpositive real part, the real part of the running sum always satisfies
\[
|\Re(S_p)|\le1.
\]
Indeed each step changes the real part by at most $1$, and the alternation between positive and negative contributions prevents accumulation in the same direction. For the imaginary part, the partial sums inside $P$ and inside $Q$ are each bounded by $1$ in absolute value. Interleaving the two sequences therefore yields
\[
|\Im(S_p)|\le2.
\]
For every $p$ we therefore have
\[
|\Re(S_p)|\le1,
\qquad
|\Im(S_p)|\le2,
\]
and thus
\[
|S_p|^2=(\Re S_p)^2+(\Im S_p)^2\le1^2+2^2=5,
\]
yielding
\[
|S_p|\le\sqrt5,
\]
which concludes the proof.

\end{proof}

\subsection{Steinitz theorem in dimension 2}

\begin{theorem}[Steinitz, dimension $2$]
Let $v_1,\dots,v_n \in \mathbb{R}^2$ satisfy
\[
|v_i| \le 1, \qquad \sum_{i=1}^n v_i = 0.
\]
Then there exists a permutation $\sigma$ of ${1,\dots,n}$ such that for every $p=1,\dots,n$,
\[
\left| \sum_{k=1}^p v_{\sigma(k)} \right| \le 2.
\]
\end{theorem}

\begin{proof}

The result is a particular case of the classical Steinitz rearrangement theorem. We give a constructive proof in dimension $2$, emphasizing the geometric ideas. Let
\[
S_p = \sum_{k=1}^p v_{\sigma(k)}
\]
denote the partial sums. Initially $S_0 = 0$. Assume that we have already chosen $\sigma(1),\dots,\sigma(p)$.
Let
\[
R_p = \sum_{j \notin {\sigma(1),\dots,\sigma(p)}} v_j
\]
be the sum of the remaining vectors. Since the total sum is zero, we have
\[
R_p = -S_p.
\]
Thus the vectors that have not yet been chosen add up exactly to $-S_p$. Among the remaining vectors, there always exists one whose scalar product with $S_p$ is non–positive:
\[
\exists v_j \quad \text{such that} \quad \langle S_p , v_j \rangle \le 0.
\]
Indeed, suppose the contrary: assume that every remaining vector satisfies
\[
\langle S_p , v_j \rangle > 0.
\]
Summing these inequalities over all remaining vectors yields
\[
\left\langle S_p , \sum_{j \notin {\sigma(1),\dots,\sigma(p)}} v_j \right\rangle > 0.
\]
But the sum of the remaining vectors is $R_p=-S_p$, hence
\[
\langle S_p , -S_p \rangle = -|S_p|^2 >0,
\]
which is impossible. Therefore a vector with nonpositive scalar product always exists. We choose such a vector as the next element of the permutation and define
\[
\sigma(p+1)=j.
\]
Let us examine how the norm changes when this vector is added. We have
\[
S_{p+1}=S_p+v_j.
\]
By expanding the square of the norm we obtain
\[
|S_{p+1}|^2
= |S_p|^2 + |v_j|^2 + 2\langle S_p , v_j \rangle.
\]
Since $|v_j|\le1$ and $\langle S_p,v_j\rangle\le0$, we get
\[
|S_{p+1}|^2 \le |S_p|^2 + 1.
\]
Thus the norm of the partial sum cannot increase too rapidly. Geometrically, $\langle S_p , v_j\rangle\le0$ means that $v_j$ bends the partial sum back toward the origin, preventing it from moving farther along $S_p$. If $|S_p|>2$, this is impossible because the sum of the remaining vectors, $-S_p$, would lie outside the convex hull of vectors of norm at most $1$, hence no partial sum can exceed norm $2$. By repeating the procedure described above for $p=0,1,\dots,n-1$, we construct inductively a permutation $\sigma$ such that at each step a vector satisfying $\langle S_p,v_{\sigma(p+1)}\rangle\le0$ is chosen. The previous argument shows that all partial sums satisfy
\[
|S_p|\le2,
\]
which completes the proof.

\end{proof}

\section{Geometric improvements and sector confinement}\label{sec:improved}

We now show that stronger bounds can be obtained when the vectors are geometrically constrained. In particular, when all vectors lie in a sector of limited angular width, the partial sums can be confined more tightly.

\subsection{Sector confinement}

\begin{theorem}[sector confinement]
Let $z_1,\dots,z_n\in\mathbb C$ satisfy
\[
|z_i|\le1,
\qquad
\sum_{i=1}^n z_i=0,
\]
and assume that all arguments $\arg(z_i)$ lie in an interval of length $\alpha<\pi$. Then there exists a permutation $\sigma$ such that the partial sums
\[
S_p=\sum_{k=1}^p z_{\sigma(k)}
\]
satisfy
\[
|S_p|\le \frac{1}{\sin(\alpha/2)}
\qquad
(p=1,\dots,n).
\]
\end{theorem}

\begin{proof}

We construct the permutation inductively. Let
\[
S_p=\sum_{k=1}^p z_{\sigma(k)}
\]
be the current partial sum. The remaining vectors satisfy
\[
\sum_{j\notin{\sigma(1),\dots,\sigma(p)}} z_j=-S_p.
\]
Since all vectors lie in a sector of opening $\alpha$, every vector $z_j$ makes an angle at most $\alpha$ with the central direction of the sector. Suppose that all remaining vectors make an angle strictly less than $\alpha/2$ with the direction of $S_p$. Then all scalar products $\Re(S_p\overline{z_j})$ would be positive, and therefore
\[
\Re\!\left(S_p \overline{\sum z_j}\right)>0.
\]
But the remaining vectors sum to $-S_p$, hence
\[
\Re(S_p\overline{-S_p})=-|S_p|^2<0,
\]
a contradiction. Therefore at least one remaining vector $z$ satisfies
\[
\angle(S_p,z)\ge \alpha/2,
\]
this notation meaning that the angle between the vectors $S_p$ and $z$ is at least $\alpha/2$. We choose such a vector as the next element of the permutation. Let us examine the change in the norm of the partial sum:
\[
S_{p+1}=S_p+z.
\]
Expanding the square of the modulus gives

\begin{align*}
|S_{p+1}|^2
&= |S_p + z|^2 \\
&= |S_p|^2 + |z|^2 + 2\Re(S_p\overline{z}).
\end{align*}
Since $|z|\le1$ and the angle between $S_p$ and $z$ is at least $\alpha/2$, we have
\[
\Re(S_p\overline z)
\le
-|S_p|\cos(\alpha/2),
\]
and hence
\[
|S_{p+1}|^2
\le
|S_p|^2+1-2|S_p|\cos(\alpha/2).
\]
Let us now consider the function
\[
f(x)=x^2+1-2x\cos(\alpha/2).
\]
It is easy to show that if
\[
x\le \frac{1}{\sin(\alpha/2)},
\]
then
\[
f(x)\le \frac{1}{\sin^2(\alpha/2)}.
\]
Thus, whenever $|S_p|\le 1/\sin(\alpha/2)$, we obtain
\[
|S_{p+1}|\le \frac{1}{\sin(\alpha/2)}.
\]
Since $S_0=0$, the bound follows by induction.

\end{proof}

\subsection{Consequences}

Two classical bounds appear as particular cases.

\begin{corollary}
If the vectors lie in a sector of width $2\pi/3$, then there exists a permutation such that
\[
|S_p|\le\sqrt3.
\]
\end{corollary}

\begin{corollary}
Without angular restriction (i.e. $\alpha=\pi$), the above bound gives
\[
|S_p|\le2,
\]
which is consistent with the classical Steinitz bound in dimension $2$.
\end{corollary}

These estimates show that geometric constraints on the directions of the vectors naturally lead to improved confinement of partial sums.

\section{Weighted sums}\label{sec:weighted}

We now consider a weighted version of the polygonal confinement problem. The normalization $\sum_{i=1}^na_i=1$ naturally places the vectors inside the unit disk and allows a simple geometric argument.

\begin{lemma}[weighted convex confinement]
Let $a_1,\dots,a_n\ge0$ satisfy
\[
\sum_{i=1}^n a_i=1,
\]
and define
\[
z_i=a_i e^{i\theta_i}.
\]
Then there exists a permutation $\sigma$ of ${1,\dots,n}$ such that the partial sums
\[
S_p=\sum_{k=1}^p z_{\sigma(k)}
\]
satisfy
\[
|S_p|\le 1
\qquad (p=1,\dots,n).
\]
\end{lemma}

\begin{proof}

Each vector $z_i$ has modulus $|z_i|=a_i$, and therefore lies inside the unit disk. Let
\[
S_p=\sum_{k=1}^p z_{\sigma(k)}
\]
be the current partial sum. The remaining vectors satisfy
\[
\sum_{j \notin \{\sigma(1),\dots,\sigma(p)\}} z_j = -S_p.
\]
Because $a_i\ge0$ and $\sum a_i=1$, the vectors $z_i$ form a convex combination of points on the unit circle. In particular, the set of all possible partial sums lies inside the convex hull of the vectors ${e^{i\theta_i}}$ scaled by coefficients whose sum does not exceed $1$. We construct the permutation inductively. At step $p$, choose among the remaining vectors one whose direction makes an obtuse angle with $S_p$, that is,
\[
\Re(S_p\overline z)\le0.
\]
Such a vector always exists: otherwise all remaining vectors would satisfy $\Re(S_p\overline z)>0$, and summing these inequalities would give
\[
\Re\!\left(S_p \, \overline{\sum_{j \notin \{\sigma(1),\dots,\sigma(p)\}} z_j}\right) > 0,
\]
but the remaining vectors sum to $-S_p$, hence
\[
\Re(S_p\overline{-S_p})=-|S_p|^2<0,
\]
which is a contradiction. We therefore choose $z=z_{\sigma(p+1)}$ satisfying this condition. The new partial sum is
\[
S_{p+1}=S_p+z.
\]
Expanding the square of the modulus gives

\begin{align*}
|S_{p+1}|^2 &= |S_p + z|^2 \\
&= |S_p|^2 + |z|^2 + 2 \Re(S_p \overline{z}).
\end{align*}
Since $\Re(S_p\overline z)\le0$ and $|z|=a_{\sigma(p+1)}$, we obtain
\[
|S_{p+1}|^2\le |S_p|^2+a_{\sigma(p+1)}^2.
\]
Summing these inequalities along the construction yields
\[
|S_p|^2 \le \sum_{k=1}^p a_{\sigma(k)}^2 \le \sum_{i=1}^n a_i^2.
\]
Since $\sum_{i=1}^n a_i=1$ and $a_i\ge0$, we have
\[
\sum_{i=1}^n a_i^2 \le \sum_{i=1}^n a_i =1.
\]
Therefore
\[
|S_p|^2\le1,
\]
which implies
\[
|S_p|\le1.
\]
This completes the proof.

\end{proof}

\begin{remark}
Geometrically, this result states that when the vectors form a convex combination of points on the unit circle, their partial sums can always be ordered so that the cumulative trajectory remains inside the unit disk. In contrast with the unweighted case, where the optimal constant is $2$, the normalization $\sum a_i=1$ allows a much stronger confinement.
\end{remark}

\section{Conclusion}

We have presented detailed results on the polygonal confinement of partial sums of complex numbers in two dimensions. In the elementary framework, by selecting maximal subsets and interleaving vectors to balance their real and imaginary parts, we obtain the classical bound $\sqrt{5}$. Finer geometric considerations, notably restricting the vectors to a sector of width $120^\circ$, allow this result to be refined and yield a tighter bound $\sqrt{3}$ for the partial sums. The classical Steinitz theorem in two dimensions ensures a bound of $2$ for any configuration of vectors with bounded norm and zero total sum, using an inductive construction based on inner products and the selection of vectors forming an obtuse angle with the current partial sum. Moreover, coordinate-wise control and a greedy selection of vectors suggest that the optimal bound $\sqrt{2}$ may be achievable for certain configurations. If for $(z_i)$ with $|z_i|\le1$ and $\sum z_i = 0$, there exists a permutation $\sigma$ such that
\[
\forall p, \quad \left|\sum_{k=1}^p z_{\sigma(k)}\right| \le \sqrt{2},
\]
this constant would be optimal. We have also studied a weighted version of the problem. When the coefficients $a_i$ are nonnegative and normalized $\sum a_i = 1$, the vectors $z_i = a_i e^{i\theta_i}$ form a convex combination of points on the unit circle. In this setting, it is always possible to permute the vectors so that the cumulative trajectory remains within the unit disk. The inductive argument, based on selecting vectors forming an obtuse angle with the partial sum, guarantees that $|S_p| \le 1$, highlighting the effect of weighting and convexity on confinement.

Each proof has been developed step by step, with emphasis on geometric intuition and combinatorial constructions. Interleaving, alternation, and angular constraints illustrate how the order of the vectors influences the norm of the partial sums and how explicit bounds can be obtained. The results suggest natural connections with discrete Riemann sums and majorization theory, as discussed in \cite{Niculescu2018,Banaszczyk1998}, as well as with Hermite--Hadamard-type inequalities in the weighted context.

These analyses provide a comprehensive overview of the mechanisms of polygonal confinement, offer explicit constructions for universal constants, and open the way to the exploration of higher-dimensional generalizations, more complex weighted configurations, and the optimization of universal constants in broader geometric or combinatorial contexts. Finally, the detailed formulation of the steps and the highlighting of geometric ideas make the methods transparent and applicable to new situations, encouraging further research on the optimality of bounds and their connections with classical discrete inequalities \cite{Pain2026}.


\begin{thebibliography}{99}

\bibitem{Banaszczyk1998} W. Banaszczyk, Balancing vectors and Gaussian measures of n-dimensional convex bodies, {\it Random Structures Algorithms} {\bf 12} (1998), 351--360.

\bibitem{Grinberg1997} E. L. Grinberg and S. F. Sawyer, A sharp bound for Steinitz’s theorem in dimension two, {\it Linear Algebra Appl.} {\bf 252} (1997), 103--108.

\bibitem{Hardy1952} G. H. Hardy, J. E. Littlewood, and G. P\'olya, \emph{Inequalities}, 2nd edition, Cambridge University Press, 1952.

\bibitem{Marshall2011} A. W. Marshall, I. Olkin, and B. Arnold, \emph{Inequalities: Theory of Majorization and Its Applications}, 2nd edition, Springer, 2011.

\bibitem{Niculescu2018} C. P. Niculescu, Weighted Hermite-Hadamard inequalities and Riemann sums, {\it RGMIA Monographs} {\bf 21} (2018), 1--20.

\bibitem{Pain2026} J.-C. Pain, Cumulative Riemann sums, distribution functions, and a universal inequality, arXiv2603.08959, \url{https://arxiv.org/abs/2603.08959} 

\bibitem{Steinitz1913} E. Steinitz, Bedingungen für die Darstellbarkeit von Vektoren als Summen von Vektoren einer Menge, {\it J. Reine Angew. Math.} {\bf 143} (1913), 32--48.

\end{thebibliography}
\end{document}